\documentclass[11pt]{amsart}
\newcommand{\Hom}{\operatorname{Hom}}
\newcommand{\Gen}{\operatorname{Gen}}
\renewcommand{\dim}{\operatorname{dim}}
\newcommand{\Ext}{\operatorname{Ext}}
\renewcommand{\mod}{\operatorname{mod}}
\newcommand{\undim}{\underline{\dim}}

\newtheorem{thm}{Theorem}
\newtheorem{lemma}[thm]{Lemma}
\newtheorem{proposition}[thm]{Proposition}
\newtheorem{corollary}[thm]{Corollary}

\theoremstyle{definition}
\newtheorem{definition}[thm]{Definition}
\newtheorem{remark}[thm]{Remark}

\usepackage[hide links]{hyperref}
\usepackage{tikz-cd}
\usepackage[margin = 1in]{geometry}

\title{A uniqueness property of $\tau$-exceptional sequences}
\author{Eric J. Hanson}
\address{(E. J. Hanson) Département de Mathématiques, Université de Sherbrooke;\newline \indent Département de Mathématiques, LACIM, Université du Québec à Montréal}
\email{ejhanso3@ncsu.edu}

\author{Hugh Thomas}
\address{(H. Thomas) Département de Mathématiques, LACIM, Université du Québec à Montréal}
\email{thomas.hugh\_r@uqam.ca}

\thanks{This version of the article has been accepted for publication, after peer review (when applicable) and is subject
to Springer Nature’s AM terms of use, but is not the Version of Record and does not reflect post-acceptance
improvements, or any corrections. The Version of Record is available online at: \url{https://doi.org/10.1007/s10468-023-10226-w}.}

\begin{document}
\maketitle

\begin{abstract}
Recently, Buan and Marsh showed that if two complete $\tau$-exceptional sequences agree in all but at most one term, then they must agree everywhere, provided the algebra is $\tau$-tilting finite. They conjectured that the result holds without that assumption. We prove their conjecture. Along the way, we also show that the dimension vectors of the modules in a $\tau$-exceptional sequence are linearly independent.
\end{abstract}

\section{Introduction}

Let $\Lambda$ be a finite-dimensional basic algebra over a field $K$, and $\mod\Lambda$ the category of finitely-generated (left) $\Lambda$-modules. A sequence $(X_1,\ldots,X_k)$ of indecomposable objects in $\mod\Lambda$ is called an \emph{exceptional sequence} if the endomorphism ring of each $X_i$ is a division algebra, 
$\Hom(X_j,X_i) = 0$ for all $1 \leq i < j \leq k$, and $\Ext^m(X_j,X_i) = 0$ for all $1 \leq i \leq j \leq k$ and $m > 0$. The exceptional sequence is said to be \emph{complete} if $k$ is equal to the number of (isomorphism classes of) simple objects in $\mod\Lambda$. We denote this number by $n$ throughout the paper.

Exceptional sequences play a particularly notable role in the study of hereditary algebras. In particular, the set of complete exceptional sequences over such an algebra admits a transitive action by the braid group on $n$ strands \cite{CB,ringel}. Complete exceptional sequences can also be used to parameterize the maximal chains in the corresponding poset of noncrossing partitions \cite{IS}. Exceptional sequences are also used to define the ``cluster morphism category'' of a hereditary algebra \cite{IT}, which in turn can be used to construct an Eilenberg-MacLane space for the corresponding ``picture group'', as defined in \cite{ITW}.

Unlike in the hereditary case, there exist non-hereditary algebras for which complete exceptional sequences do not exist \cite[Introduction]{BM_exceptional}. To address this problem, Buan and Marsh used the $\tau$-tilting theory of \cite{AIR}, and specifically the $\tau$-tilting reduction of \cite{jasso, DIRRT}, to define the two closely related notions of \emph{$\tau$-exceptional sequences} and \emph{signed $\tau$-exceptional sequences} (see Definition~\ref{def:tau_ex}) in \cite{BM_exceptional}. Importantly, it is an immediate consequence of the definition that (signed) $\tau$-exceptional sequences can always be ``completed'', in the sense discussed immediately after Definition~\ref{def:tau_ex}. This development has resulted in several recent papers generalizing the relationships between exceptional sequences, noncrossing partitions, cluster morphism categories, and picture groups to their ``$\tau$-tilting theoretic analogs'', for example \cite{BaH,borve,BM_wide,BuH,HI,MT}.

In the recent paper \cite{BM_mutation}, Buan and Marsh begin to investigate the question of whether the set of complete signed $\tau$-exceptional sequences admits a (transitive) braid group action. As a step on the way, they prove that if $\Lambda$ is $\tau$-tilting finite (in the sense of \cite{DIJ}), then there do not exist complete $\tau$-exceptional sequences which ``differ in only one place''. They also conjecture that this result holds without the assumption of $\tau$-tilting finiteness. This generalizes \cite[Lemma~2]{CB}, where the analogous statement is shown for complete exceptional sequences over hereditary algebras with no finiteness assumptions. In this paper, we prove Buan and Marsh's conjecture as Theorem~\ref{thm:main} below. This also implies uniqueness results for the signed $\tau$-exceptional sequences of \cite{BM_exceptional} and the \emph{brick-}$\tau$-exceptional sequences of \cite{BaH}, see Definition~\ref{def:tau_ex} and Corollary~\ref{cor:main}. We work in the generality of elementary algebras, which in particular includes all cases where the field $K$ is algebraically closed.

Our proof is based upon King's theory of stability conditions \cite{king}, and specifically the relationship between semistability and Jasso's $\tau$-perpendicular categories established in \cite[Proposition~3.13]{BST}. (We review the relevant definitions below.) This allows us to show that the dimension vectors of the objects in a $\tau$-exceptional sequence are linearly independent (Proposition~\ref{linindep}(1)), from which we deduce our main result. This linear independence also holds for signed $\tau$-exceptional sequences and brick-$\tau$-exceptional sequences (Proposition~\ref{linindep}(2,3)).

\subsection*{Acknowlegements}
The authors are grateful to Aslak Bakke Buan for helpful discussions.
H.T. was partially supported by an NSERC Discovery Grant and the Canada Research Chairs program. Both authors are grateful to the Centre for Advanced Study at the Norwegian Academy of Science and Letters in Oslo for the excellent working conditions they provided. 

\section{Definitions}

Let $X \in \mod\Lambda$ be a basic module, and denote by $\mathrm{rk}(X)$ the number of indecomposable direct summands of $X$. The module $X$ is said to be \emph{$\tau$-rigid} if $\Hom(X,\tau X) = 0$, where $\tau$ denotes the Auslander-Reiten translation. The \emph{$\tau$-perpendicular category} of $X$ is the subcategory
$$\mathcal{J}(X) = {}^\perp \tau X\cap X^\perp 
= \{Y \in \mod\Lambda \mid \Hom(X,Y) = 0 = \Hom(Y,\tau X)\}.$$
By \cite[Theorem~4.12]{jasso}, there exists a finite-dimensional algebra $\Lambda'$ with $n-\mathrm{rk}(X)$ simple modules which admits an exact equivalence $\mathcal{J}(X) \approx \mod \Lambda'$.

Before defining $\tau$-exceptional sequences and their variants, we recall several pieces of notation. For $X \in \mod\Lambda$, we denote by $X[1]$ the \emph{shift} of $X$. This notation and terminology refers to the chain complex in the bounded derived category $\mathcal{D}^b(\mod\Lambda)$ which consists only of the module $X$ concentrated in degree -1, but for the purposes of this paper it is sufficient to view $X[1]$ as a ``tagged'' version of the module $X$. We then denote $\mathcal{C}(\mod\Lambda) = \mod\Lambda \sqcup \mod\Lambda[1]$, and for $X \in \mod\Lambda$, we denote
$|X[1]| = |X| = X.$
Finally, for $X \in \mod\Lambda$ indecomposable, let $\mathrm{rad}(X,X)$ denote the set of non-invertible endomorphisms of $X$. We then denote
$$\beta(X) = X/\sum_{f \in \mathrm{rad}(X,X)} \mathrm{im}(f) = X/\mathrm{rad}_{\mathrm{End}(X)}X.$$
Recall that a module $Y$ is called a \emph{brick} if every nonzero endomorphism of $Y$ is invertible, and is called an \emph{f-brick} if an addition the smallest torsion class containing $Y$ is functorially finite. (The definitions of torsion classes and functorial finiteness are not needed for this paper.) The \emph{brick-$\tau$-rigid correspondence} of \cite{DIJ} then says that the association $X \mapsto \beta(X)$ induces a bijection from the set of indecomposable $\tau$-rigid modules to the set of f-bricks.

\begin{definition}\label{def:tau_ex}\
    \begin{enumerate}
        \item \cite[Definition~1.3]{BM_exceptional} A sequence $(X_1,\ldots,X_k)$ of indecomposable objects in $\mod\Lambda$ is called a \emph{$\tau$-exceptional sequence} if $X_k$ is $\tau$-rigid and $(X_1,\ldots,X_{k-1})$ is a $\tau$-exceptional sequence in $\mathcal{J}(X_k)$, identified with the module category $\mod\Lambda'$ as above.
        \item \cite[Definition~1.3]{BM_exceptional} A sequence $(\mathcal{U}_1,\ldots,\mathcal{U}_k)$ of objects in $\mathcal{C}(\mod\Lambda)$ is called a \emph{signed $\tau$-exceptional sequence} if (i) one of (i.a) $\mathcal{U}_k \in \mod\Lambda$ and is $\tau$-rigid or (i.b) $|\mathcal{U}_k|$ is projective, and (ii) $(\mathcal{U}_1,\ldots,\mathcal{U}_{k-1})$ is a signed $\tau$-exceptional sequence in $\mathcal{C}(\mathcal{J}(|U_k|))$, identified with the category $\mathcal{C}(\mod\Lambda')$ for $\Lambda'$ as above.
        \item A sequence $(X_1,\ldots,X_k)$ of objects in $\mod\Lambda$ is called a \emph{brick-$\tau$-exceptional sequence} if there exists a $\tau$-exceptional sequence $(Y_1,\ldots,Y_k)$ such that $X_i = \beta(Y_i)$ for all $1 \leq i \leq k$.
    \end{enumerate}
    In all three cases, we say that the sequence is \emph{complete} if $k = n$.
\end{definition}

Any signed $\tau$-exceptional sequence $(\mathcal{U}_1,\ldots,\mathcal{U}_k)$  can be completed to a complete signed $\tau$-ex\-cep\-tion\-al sequence $(\mathcal{V}_1,\ldots,\mathcal{V}_{n-k},\mathcal{U}_1,\ldots,\mathcal{U}_k)$, by combining \cite[Theorem 5.4]{BM_exceptional} and \cite[Theorem 2.10]{AIR}.
Note that, unlike for exceptional sequences over hereditary algebras, it is not in general true that a signed $\tau$-exceptional sequence can be completed by inserting modules at arbitrarily specified positions in the sequence.
These same statements follow for unsigned and brick-$\tau$-exceptional sequences, as well.

\begin{remark}\
    \begin{enumerate}
        \item It follows immediately from the definition that if $(\mathcal{U}_1,\ldots,\mathcal{U}_k)$ is a signed $\tau$-exceptional sequence, then $(|\mathcal{U}_1|,\ldots,|\mathcal{U}_k|)$ is a $\tau$-exceptional sequence. On the other hand, allowing for the (relatively) projective objects in a $\tau$-exceptional sequence to be ``tagged'' is critical in defining the ($\tau$-)cluster morphism categories of \cite{BuH,BM_wide,IT}.
        \item By \cite[Theorem~1.4]{DIJ}, the map $\beta$ induces a bijection between the set of $\tau$-exceptional sequences and the set of brick-$\tau$-exceptional sequences.
        Note, however, that the inverse of this bijection is defined more subtly as the inverse of the brick-$\tau$-rigid correspondence $\beta$ generally changes when one restricts to a subcategory of the form $\mathcal{J}(X)$. See \cite[Remark~3.19]{BaH} for an example.
        \item For further motivation of the notion of brick-$\tau$-exceptional sequences, see \cite{BaH}. Specifically, we note that brick-$\tau$-exceptional sequences appear in \cite[Definition~5.7]{BaH} without being given an explicit name.
    \end{enumerate}
\end{remark}

We conclude this section by recalling the notions of $g$-vectors and semistability. Choose some indexing $P(1),\ldots,P(n)$ of the indecomposable projectives in $\mod\Lambda$. For $X \in \mod\Lambda$, let $P_1^X \rightarrow P_0^X \rightarrow X \rightarrow 0$ be a minimal projective presentation of $X$, and decompose
$$P_0^X = \bigoplus_{i = 1}^n P(i)^{\oplus p_i},\qquad\qquad P_1^X = \bigoplus_{i = 1}^n P(i)^{\oplus m_i}.$$
The \emph{$g$-vector} of $X$ is then defined as
$$g_X = \sum_{i = 1}^n (p_i-m_i)\cdot e_i \in \mathbb{R}^n,$$
where $e_i$ denotes the $i$-th standard basis vector. Similarly, the \emph{dimension vector} of $X$ is defined as
$$\undim X = \sum_{i = 1}^n \left(\dim\Hom(P(i),X)\right)\cdot e_i.$$
We note that while we consider $g_X, \undim X \in \mathbb{R}^n$, both vectors are integer-valued. We also define $\undim X[1] = -\undim X$.

\begin{definition}\cite{king}
    Let $\theta \in \mathbb{R}^n$ and let $X \in \mod\Lambda$. Then $X$ is \emph{$\theta$-semistable} if (a) $\theta\cdot \undim X = 0$ and (b) $\theta \cdot \undim X' \leq 0$ for all submodules $X' \subseteq X$. (There are two possible choices of sign convention here, and both are used in the literature.).
\end{definition}

\section{Proof of main result}

In order to prove our main result, we need some lemmas. Because the proofs are
short and instructive, we have preferred to include them.

For $X \in \mod\Lambda$ we denote by $\Gen X$ the full subcategory of $\mod\Lambda$ consisting of those modules $M$ for which there exists a surjective map $X^{\oplus r} \rightarrow M$ for some positive integer $r$.

\begin{lemma} \label{nohom} If $X$ is $\tau$-rigid, then $\Hom(\Gen X,\tau X)=0$. \end{lemma}

\begin{proof} Let $M\in \Gen X$. If $X = 0$ then we are done. Otherwise, there is a surjective map from $X^{\oplus r}$ to $M$ for some positive $r$. Now a non-zero morphism from $M$ to $\tau X$ induces a non-zero morphism from $X$ to $\tau X$, but that is impossible by hypothesis.
\end{proof}

The following result can be found in \cite[Theorem~1.4a]{AR}. We give a proof based on that of \cite[Proposition~5.3]{AIR}.

\begin{lemma} \label{interp} Let $g_X$ be the $g$-vector of $X$. Let $L$ be
  a module. Then:
  $$g_X\cdot \undim L  = \dim \Hom(X,L)-\dim\Hom(L,\tau X).$$
\end{lemma}

\begin{proof}
    Let $P_1^X \xrightarrow{f} P_0^X \xrightarrow{q} X \rightarrow 0$ be a minimal projective presentation of $X$. By definition, there is a short exact sequence $0 \xrightarrow{\iota} \tau X \rightarrow \nu P_1^X \xrightarrow{\nu f} \nu P_0^X$, where $\nu$ denotes the Nakayama functor. Denoting by $D$ the $K$-linear dual, we then have a commutative diagram of short exact sequences
    $$
    \begin{tikzcd}[column sep = 1em]
        0 \arrow[r] & \Hom(L, \tau X) \arrow[r,"\iota_*"] & \Hom(L, \nu P_1^X) \arrow[r,"(\nu f)_*"] \arrow[d,leftrightarrow,"\wr"] & \Hom(L,\nu P_0^X\arrow[d,leftrightarrow,"\wr"])\\
        && D\Hom(P_1^X,L) \arrow[r,"Df^*"] & D\Hom(P_0^X,L) \arrow[r,"Dq^*"] & D\Hom(X,L) \arrow[r] & 0.
    \end{tikzcd}
    $$
    Taking the alternating sum of dimensions over the induced 4-term exact sequence, this gives
    $$\dim\Hom(L,\tau X) - \dim\Hom(P_1^X,L) + \dim\Hom(P_0^X, L) - \dim\Hom(X,L) = 0.$$
    On the other hand, we have that
    $$g_X \cdot \undim L = \dim\Hom(P_0^X,L) - \dim\Hom(P_1^X,L)$$
    by the definition of $g_X$. This proves the result.
\end{proof}

\begin{lemma}\cite[Proposition~3.13]{BST}\label{equiv} Let $X$ be indecomposable $\tau$-rigid. Then
$M$ is $g_X$-semistable if and only if $M\in \mathcal{J}(X)$.
\end{lemma}

\begin{proof} 
Suppose $M$ is semistable with respect to $g_X$.
In particular, by Lemma~\ref{interp}, $\dim\Hom(X,M)-\dim\Hom(M,\tau X)=0$. Suppose that
$\Hom(X,M)\ne 0$. So $M$ has a non-zero subobject $L$ in $\Gen X$.
By Lemma \ref{nohom}, $\Hom(L,\tau X)=0$, However, $\Hom(X,L)$ is nonzero.
So $g_X\cdot \undim L > 0$. This violates semistability of $M$.

Now suppose $M\in \mathcal{J}(X)$. It follows immediately that $g_X\cdot \undim M=0$.
Let $L$ be a subobject of $M$. Now $\Hom(X,L)$  is necessarily 0,
so $g_X\cdot \undim L$ is non-positive. Thus, $M$ is $g_X$-semistable.
\end{proof}

\begin{proposition} \label{linindep}\
\begin{enumerate}
    \item The dimension vectors of the modules in a $\tau$-exceptional
    sequence are linearly independent.
    \item The dimension vectors of the objects in a signed $\tau$-exceptional sequence are linearly independent.
    \item The dimension vectors of the modules in a brick-$\tau$-exceptional sequence are linearly independent.
\end{enumerate}

\end{proposition}

\begin{proof}
(1,3) It suffices to prove the results for complete sequences. Let $(X_1,\ldots,X_n)$ be a complete $\tau$-exceptional sequence and  $(\beta(X_1),\ldots,\beta(X_n))$ the corresponding complete brick-$\tau$-exceptional sequence. The proof is by induction on $n$. Both statements
 certainly hold for $n=1$. Suppose that
the results are known for complete (brick-)$\tau$-exceptional sequences of length
$n-1$. Since $X_n$ is $\tau$-rigid, $g_{X_n}$
is nonzero. The modules $X_i$ for $1 \leq i \leq n-1$ lie in $\mathcal{J}(X_n)$, so their dimension vectors lie in the orthogonal
complement of $g_{X_n}$ by Lemma \ref{equiv}. It follows from the definition (together with the Auslander--Reiten formulas) that $\mathcal{J}(X_n)$ is closed under cokernels. Thus the modules $\beta(X_i)$ for $1 \leq i \leq n-1$ also lie in $\mathcal{J}(X_n)$, so their dimension vectors also lie in the orthogonal complement of $g_{X_n}$. By the induction hypothesis,
both sets of dimension vectors span this space. Now, Lemma \ref{interp} says that $g_{X_n} \cdot \dim X_n=
\dim \Hom(X_n,X_n)-\dim\Hom(X_n,\tau X_n)>0$. Thus, $\undim X_n$ does not lie in the
orthogonal complement of $g_{X_n}$. Similarly, Lemmas~\ref{nohom} and~\ref{interp} imply that $g_{X_n} \cdot \beta(X_n) > 0$, and so $\undim \beta(X_n)$ does not lie in the
orthogonal complement of $g_{X_n}$. It follows that both $\{\undim X_i\}_{i = 1}^n$ and $\{\undim \beta(X_i)\}_{i = 1}^n$ span $\mathbb{R}^n$ as desired.

(2) Let $(\mathcal{U}_1,\ldots,\mathcal{U}_k)$ be a signed $\tau$-exceptional sequence. Then $(|\mathcal{U}_1|,\ldots,|\mathcal{U}_k|)$ is a $\tau$-exceptional sequence, and so the dimension vectors of these modules are independent. The result thus follows from the fact that $\undim X[1] = -\undim X$ for $X \in \mod\Lambda$.
\end{proof}

We now prove our main result.

\begin{thm}\label{thm:main}
        Let $(X_1,\ldots,X_n)$ and $(Y_1,\ldots,Y_n)$ be  complete $\tau$-exceptional sequences. If there exists an index $j$ such that $X_i \cong Y_i$ for $i \neq j$, then also $X_j \cong Y_j$.
\end{thm}

\begin{proof}
    We prove the result by induction on $n$. If $n = 1$, then $\mod\Lambda$ contains a unique basic $\tau$-rigid module (namely $P(1)$), and there is nothing to show. Thus suppose the result holds for complete $\tau$-exceptional sequences of length $n-1$. If $X_n \cong Y_n$, we then have that $(X_1,\ldots,X_{n-1})$ and $(Y_1,\ldots,Y_{n-1})$ are complete $\tau$-exceptional sequences in $\mathcal{J}(X_n) = \mathcal{J}(Y_n)$ which satisfy the hypothesis of Theorem~\ref{thm:main}, and so the result follows from the induction hypothesis. It remains only to consider the case $j = n$.

  To simplify notation, we identify $U_i := X_i = Y_i$ for $i < n$. We know that for $1\leq i \leq n-1$, we have $g_{X_n} \cdot \undim U_i=0$, and
  the $\undim U_i$ are linearly independent by Proposition \ref{linindep}. Thus $g_{X_n}$ is necessarily in the
  one-dimensional orthogonal complement to the span of the $\undim U_i$.  The
  same thing is true for $g_{Y_n}$. Thus, $g_{X_n}$ and $g_{Y_n}$ are scalar multiples
  of each other.

If they are positive scalar multiples then we are done, as $\tau$-rigid modules are determined by their $g$-vectors. Indeed, if there exist $c, d \in \mathbb{N}$ such that $cg_{X_n} = dg_{Y_n}$, then the (non-basic) $\tau$-rigid modules $(X_n)^{\oplus c}$ and $(Y_n)^{\oplus d}$ have the same $g$-vector. By \cite[Theorem~6.5]{DIJ}, this means $(X_n)^{\oplus c} \cong (Y_n)^{\oplus d}$, and so $X_n \cong Y_n$ since both modules are indecomposable.

Consider now the case that $g_{Y_n}$ is a negative scalar multiple of $g_{X_n}$.
In this case, each $U_i$ is $g_{X_n}$- and $(-g_{X_n})$-semistable, which means all the
dimension vectors of subobjects of each $U_i$ are orthogonal to $g_X$.
Thus the
same is true for all subquotients of each $U_i$, and, in particular, for all the
simple modules appearing as composition factors of each $U_i$. At least $n-1$ simple modules must appear as composition factors, by Proposition \ref{linindep}.
This implies that $\bigoplus_{i=1}^{n-1} U_i$  is supported on $n-1$ vertices, while $g_{X_n}$ and $g_{Y_n}$
are scalar multiples of the standard basis vector corresponding to the
missing vertex. The only multiple of a standard basis vector which is the
$g$-vector of an indecomposable $\tau$-rigid module is the standard basis vector itself,
so it is not
possible that $g_{X_n}$ and $g_{Y_n}$ are negative scalar multiples, ruling out this case. 
This ends the proof.
\end{proof}

\begin{corollary}\label{cor:main}\
    \begin{enumerate}
        \item Let $(\mathcal{U}_1,\ldots,\mathcal{U}_n)$ and $(\mathcal{V}_1,\ldots,\mathcal{V}_n)$ be complete signed $\tau$-exceptional sequences. If there exists an index $j$ such that $|\mathcal{U}_i| \cong |\mathcal{V}_i|$ for $i \neq j$, then also $|\mathcal{U}_j| \cong |\mathcal{V}_j|$.
        \item Let $(X_1,\ldots,X_n)$ and $(Y_1,\ldots,Y_n)$ be a brick-$\tau$-exceptional sequences. If there exists an index $j$ such that $X_i \cong Y_i$ for $i \neq j$, then also $X_j \cong Y_j$.
    \end{enumerate}
\end{corollary}

\begin{proof}
    (1) This is an immediate consequence of Theorem~\ref{thm:main} and the fact that $(|\mathcal{U}_1|,\ldots,|\mathcal{U}_n|)$ and $(|\mathcal{V}_1|,\ldots,|\mathcal{V}_n|)$ are both complete $\tau$-exceptional sequences.

    (2) We prove the result by induction on $n$. If $n = 1$, then the only brick-$\tau$-exceptional sequence in $\mod\Lambda$ is $(S(1))$, where $S(1)$ is the simple top of the projective $P(1)$. Thus suppose the result holds for complete brick-$\tau$-exceptional sequences of length $n-1$.

    Let $(\overline{X_1},\ldots,\overline{X_n})$ and $(\overline{Y_1},\ldots,\overline{Y_n})$ be the (unique) $\tau$-exceptional sequences which satisfy $\beta(\overline{X_i}) = X_i$ and $\beta(\overline{Y_i}) = Y_i$ for all $1 \leq i \leq n$. Then $\overline{X_n} \cong \overline{Y_n}$ if and only if $X_n \cong Y_n$ by \cite[Theorem~1.4]{DIJ}. Now if $j \neq n$, then $(X_1,\ldots,X_{n-1})$ and $(Y_1,\ldots,Y_{n-1})$ are brick-$\tau$-exceptional sequences in $\mathcal{J}(\overline{X_n})$ which satisfy the hypotheses of Theorem~\ref{thm:main}, so the result follows from the induction hypothesis in this case. Thus consider the case $j = n$. In this case, we can replace $X_n$ with $\overline{X_n}$ and $Y_n$ with $\overline{Y_n}$ in the proof of Theorem~\ref{thm:main} (for the case $j = n$) to deduce that $\overline{X_n} \cong \overline{Y_n}$. This concludes the proof.
\end{proof}

\begin{corollary}\label{cor:unique_wide}\
    \begin{enumerate}
        \item Let $(X_1,\ldots,X_n)$ be a complete $\tau$-exceptional sequence. Then for each $1 \leq j \leq n$ there exists a unique $\tau$-perpendicular subcategory $\mathcal{W}_j \subseteq \mod\Lambda$ such that $(X_1,\ldots,X_j)$ is a complete $\tau$-exceptional sequence in $\mathcal{W}_j$.
        \item Let $(\mathcal{U}_1,\ldots,\mathcal{U}_n)$ be a complete signed $\tau$-exceptional sequence. Then for each $1 \leq j \leq n$ there exists a unique $\tau$-perpendicular subcategory $\mathcal{W}_j \subseteq \mod\Lambda$ such that $(\mathcal{U}_1,\ldots,\mathcal{U}_j)$ is a complete signed $\tau$-exceptional sequence in $\mathcal{W}_j$.
        \item Let $(X_1,\ldots,X_n)$ be a complete brick-$\tau$-exceptional sequence. Then for each $1 \leq j \leq n$ there exists a unique $\tau$-perpendicular subcategory $\mathcal{W}_j \subseteq \mod\Lambda$ such that $(X_1,\ldots,X_j)$ is a complete brick-$\tau$-exceptional sequence in $\mathcal{W}_j$.
    \end{enumerate}
\end{corollary}

\begin{proof}
    (1) We prove the results by reverse induction on $j$. For $j = n$, we note that the only $\tau$-perpendicular subcategory of $\mod\Lambda$ which contains $n$ simple modules is $\mod\Lambda = \mathcal{J}(0)$. Thus suppose the result holds for $j+1 \leq n$. Then there exists a unique $\tau$-perpendicular subcategory $\mathcal{W}_{j+1}$ such that $(X_1,\ldots,X_{j+1})$ is a complete $\tau$-exceptional sequence in $\mathcal{W}_{j+1}$. It is shown in \cite[Theorem~6.4]{BuH} that any $\tau$-perpendicular subcategory of $\mathcal{W}_{j+1}$ is also a $\tau$-perpendicular subcategory of $\mod\Lambda$. Thus since $\mathcal{W}_{j+1}$ is equivalent to $\mod\Lambda'$ for some finite-dimensional algebra $\Lambda'$, we can thus suppose without loss of generality that $j = n-1$ and $\mathcal{W}_{j+1} = \mod\Lambda$.

    By definition, we know that $(X_1,\ldots,X_{n-1})$ is a complete $\tau$-exceptional sequence in the $\tau$-perpendicular subcategory $\mathcal{J}(X_n)$. Thus let $Y \in \mod\Lambda$ be a basic $\tau$-rigid module and suppose that $(X_1,\ldots,X_{n-1})$ is a complete $\tau$-exceptional sequence in $\mathcal{J}(Y)$. In particular, this means $\mathcal{J}(Y)$ contains $n-1$ simple objects, and so $Y$ is indecomposable. It follows that $(X_1,\ldots,X_{n-1},Y)$ is a complete $\tau$-exceptional sequence, and so $X \cong Y$ by Theorem~\ref{thm:main}. This implies that $\mathcal{J}(X) = \mathcal{J}(Y)$, as desired.

    (2) This follows immediately from (1) and the fact that $(|\mathcal{U}_1|,\ldots,|\mathcal{U}_n|)$ is a complete $\tau$-exceptional sequence.

    (3) This follows from making slight modifications to the second paragraph of the proof of (1). Namely, we replace $\mathcal{J}(X_n)$ with $\mathcal{J}(\overline{X_n})$, where $\overline{X_n}$ is the unique indecomposable $\tau$-rigid module which satisfies $\beta(\overline{X_n}) = X_n$. We likewise replace $(X_1,\ldots,X_{n-1},Y)$ with $(X_1,\ldots,X_{n-1},\beta(Y))$ and then apply the brick-$\tau$-exceptional sequence version of Theorem~\ref{thm:main}.
\end{proof}

\end{document}